\documentclass[12pt,twoside]{amsart} 
 \title{On the anti-canonical ring and varieties of Fano type}
 \author{Paolo Cascini}
 \address{Department of Mathematics, Imperial College London, 180 Queen's Gate, London SW7 2AZ, UK}
\email{p.cascini@imperial.ac.uk}
\author{Yoshinori  Gongyo}
\address{Graduate School of Mathematical Sciences, 
the University of Tokyo, 3-8-1 Komaba, Meguro-ku, Tokyo 153-8914, Japan.}
\email{gongyo@ms.u-tokyo.ac.jp}
\address{Department of Mathematics, Imperial College London, 180 Queen's Gate, London SW7 2AZ, UK}
\email{y.gongyo@imperial.ac.uk}
\date{\today, version 0.17}
\thanks{The first author was partially supported by an EPSRC grant. The second author was partially supported by Aid-in-Grant from JSPS and Research expense from the JRF fund.}
\dedicatory{Dedicated to Professor~Fumio~Sakai on the~occasion of his retirement.}

\usepackage{pxfonts}

\usepackage{latexsym}
\usepackage{amsmath}
\usepackage{amssymb}
\usepackage{amsthm}
\usepackage{amscd}
\usepackage{enumerate} 
\usepackage{amssymb} 
\usepackage{mathrsfs}          
\usepackage[all]{xy}
\usepackage{color}
\usepackage{hyperref}

\newtheorem{thm}{Theorem}[section]

\newtheorem{prop}[thm]{Proposition}

\newtheorem{lem}[thm]{Lemma}

\newtheorem{rem}[thm]{Remark}

\newtheorem{cor}[thm]{Corollary}

\newtheorem{conj}[thm]{Conjecture}

\newtheorem{ques}[thm]{Question}

\theoremstyle{definition}

\newtheorem{defi}[thm]{Definition}

\newtheorem{eg}[thm]{Example}

\newtheorem*{ack}{Acknowledgments} 

\begin{document}
\bibliographystyle{amsalpha}
 
 \begin{abstract}We study the anti-canonical ring of a projective variety and we characterise varieties of log Fano type depending on the singularities of these models. 
 \end{abstract}

 \maketitle

 \section{Introduction}  
Over the last few years, there have been many developments  about the study of the canonical model of a projective variety. On the other hand, not much is known about its anti-canonical model. 
Few decades ago, Fumio Sakai considered this problem for projective surfaces (e.g. see 

\cite{sakai1},  and \cite{sakai2}). In particular, he gave a characterisation of the anti-canonical model of any rational surface with big anti-canonical divisor.  

The purpose of these notes is to study the anti-canonical model for  higher dimensional projective varieties.
More specifically, we describe a characterisation of  varieties of Fano type depending on the singularities of their anti-canonical models. 

Our main result is the following: 
 
 \begin{thm}[{=Corollary \ref{char by lt-ness}}]\label{main-intro1}\label{t-intro}
 Let $X$ be a $\mathbb{Q}$-Gorenstein normal variety.

Then $X$ is of Fano type if and only if   $-K_X$ is big, $R(-K_X)$ is finitely generated, and $\mathrm{Proj}\, R(-K_X)$ has log terminal singularities.
 \end{thm} 
 Recall that  a normal projective  variety  $X$ is said to be of {\em Fano type} if it admits a $\mathbb Q$-divisor $\Delta\ge 0$ such that $(X,\Delta)$ is  klt and $-(K_X+\Delta)$ is ample.  In dimension two, the main result of  \cite{big} implies that a rational surface with big anti-canonical divisor is a Mori dream space. Thus, the theorem follows from the arguments in \cite{gost}.   On the other hand, it is known that  higher dimensional varieties with big and nef anti-log canonical divisor are not necessarily Mori dream spaces (e.g. see \cite[Example 5.2]{g1}). Note that we do not assume that $X$ is a Mori dream space in Theorem \ref{t-intro}. Note that in dimension two, Hwang and Park   gave a different proof of  Theorem \ref{t-intro} (see \cite[Theorem 1.1]{hp})
 
The paper is organised as follows. We give a proof of Theorem \ref{t-intro} in Section \ref{s_proof}. In Section \ref{lc}, we consider the problem about the finite generation of Cox rings of log canonical log Fano varieties  and in Section \ref{questions} we describe some examples of varieties  of Fano type which give a negative answer to some questions raised before on these problems.

\begin{ack} The second author discussed several problems related to the examples in Section \ref{questions} with  Fr\'ed\'eric Campana, Hiromichi Takagi, Takuzo Okada and Kiwamu Watanabe while he was a Ph.D  student. 
Some of the questions considered in this paper were raised during an AIM workshop in May 2012.
We would also like to thank  Ivan Cheltsov,  Anne-Sophie Kaloghiros,  Ana-Maria Castravet,  James $\mathrm{M^{c}}$Kernan, John Christian Ottem and Vyacheslav Vladimirovich Shokurov for several useful discussions. We are particularly thankful to  Ana-Maria Castravet for pointing out an error  in a previous version of this notes. We also would like to thank DongSeon Hwang for informing us that Theorem \ref{t-intro} was already  known in dimension two \cite{hp}.
\end{ack}

  \section{Preliminaries } 
 
We work over the field of complex numbers. We use the standard definitions for the  singularities appearing in the minimal model program: 

\begin{defi}[{cf.~\cite[Definition 2.34]{komo}}]\label{sing of pairs}
Let $X$ be a normal variety  and $\Delta$ be an effective $\mathbb{Q}$-divisor on $X$ such that $K_X+\Delta$ is $\mathbb{Q}$-Cartier. 
Let $\pi: \widetilde{X} \to X$ be a birational morphism from a normal variety $\widetilde{X}$.
Then we can write 
$$K_{\widetilde{X}}=\pi^*(K_X+\Delta)+\sum_{E}a(E, X, \Delta) E,$$ 
where $E$ runs through all the distinct prime divisors on $\widetilde{X}$ and  $a(E, X, \Delta)$ is a rational number. 
We say that the pair $(X, \Delta)$ is \textit{log canonical} (resp. \textit{klt}) if $a(E, X, \Delta) \ge -1$ (resp. $a(E, X, \Delta) >-1$) for every prime divisor $E$ over $X$.  
If $\Delta=0$, we simply say that $X$ has only log canonical singularities (resp. log terminal  singularities). We say that the pair $(X,\Delta)$ is \textit{dlt} if the coefficients of $\Delta$ are less or equal one and $a(E,X,\Delta)>-1$ for any exceptional divisor $E$ of $\pi$. 
\end{defi}

 \subsection{Section rings and Cox rings}
 \begin{defi}\label{section ring}Let $X$ be a normal variety and let $L_i$ be a line bundle for $i=1, \dots k$. We define the ring 
 
 $$R(X; L_1, \dots, L_k)=\sum_{(m_1,\dots, m_k) \in \mathbb{Z}}H^{0}(X, \otimes_{i=1}^kL_i^{m_i}).
 $$ 
 In particular, if $\{L_i\}$ is a basis of $\mathrm{Pic}_{\mathbb{Q}}\,X$, 
the ring $R(X; L_1, \dots, L_k)$ is called the {\em Cox ring} of $X$ (see \cite[Subsection 2.4]{gost} for  more details). 
 \end{defi}
 \subsection{Mori dream spaces}
We now recall few basic fats about Mori Dream Spaces, introduced by Hu and Keel in \cite{hk} (see also \cite{mac-mori}). 
\begin{defi}\label{Mori dream space} 
A normal projective variety $X$ over a field is called a \textit{$\mathbb{Q}$-factorial Mori dream space} (or \textit{Mori dream space} for short) 
if $X$ satisfies the following three conditions:
\begin{enumerate}[(i)]
\item $X$ is $\mathbb{Q}$-factorial, $\mathrm{Pic}{(X)}$ is finitely generated,
and $\mathrm{Pic}{(X)}_{\mathbb{Q}} \simeq \mathrm{N}^1{(X)}_{\mathbb{Q}},$\label{fin_pic}
\item $\mathrm{Nef}{(X)}$ is the affine hull of finitely many semi-ample
line bundles, 
\item there exists a finite collection of small birational maps $f_i: X \dasharrow X_i$
such that each $X_i$ satisfies (i) and (ii), and that $\mathrm{Mov}{(X)}$ is the union
of the $f_i^*(\mathrm{Nef}{(X_i)})$.
\end{enumerate}
\end{defi}

\begin{rem}
Over the field of complex numbers,
the finite generation of $\mathrm{Pic}{(X)}$ is equivalent to the condition
$\mathrm{Pic}{(X)}_{\mathbb{Q}} \simeq \mathrm{N}^1{(X)}_{\mathbb{Q}}$.
\end{rem}
As suggested by the name, one of the advantages of working on a Mori dream space, is that on these varieties we can run an MMP for any divisor: 
\begin{prop}{ (\cite[Proposition 1.11]{hk})}\label{mmp on mds}
Let $X$ be a $\mathbb{Q}$-factorial Mori dream space. 
Then for any divisor $D$ on $X$ which is not nef, there exists always a $D$-negative contraction or a flip. In addition, any sequence of such flips terminates. 
\end{prop}
We will refer to such a sequence as a $D$-MMP.

The following known result is a useful source of  Mori dream spaces:

\begin{prop}[cf. {\cite[Theorem 3.2]{brown-Fano}}]\label{split-mds}
Let $X$ be a $\mathbb Q$-factorial projective variety which is a Mori dream space and let $L_1,\dots,L_m$ be Cartier divisors on $X$. 
Let 
$$\mathcal E=\bigoplus_{i=1}^m\mathcal O_X (L_i)$$
and let $f\colon Y=\mathbb P(\mathcal E)\to X$ be the corresponding projective bundle. 

Then $Y$ is also a Mori dream space.
\end{prop}

\begin{proof} Let $\{H_i\}_{i=1}^{k}$ be a basis of $\mathrm{Pic}_{\mathbb{Z}}\,X$. 
Then we can write 
$$L_i=\sum a_{i,j}H_j,$$
where $a_{i,j}$  are integers. Let $S$ be the tautological section of $\mathcal E$, i.e. the section corresponding to $\mathcal{O}_Y(1).$ Then for any $b_0, \dots, b_k \in \mathbb{Z}$, it follows that
$$\begin{aligned}
H^0(Y,b_0S+\sum_j b_j f^* H_j) \simeq &\bigoplus_{d_1+\cdots +d_{m}=b_0,\ d_j\geq 0}H^0(X,  \sum_{j}{d_j}L_j +  \sum_{i} b_iH_i )t_1^{d_1}\cdots t_{m}^{d_{m}}\\
&=\bigoplus_{d_1+\cdots +d_{m}=b_0,\ d_j\geq 0}H^0(X,  \sum_{i} (b_i+\sum_j a_{i,j}) H_i )t_1^{d_1}\cdots t_{m}^{d_{m}}
\end{aligned}
$$
when $b_0 \geq0$ and
$$H^0(Y,b_0S+\sum_{j=1}^{k} b_j f^* H_j)=0$$
when  $b_0 <0$.

Thus 
$R(Y;S,f^*H_1, \dots, f^*H_k )$
is  a splitting subring of 
$$R(X; H_1, \dots, H_k )[t_1,\dots, t_{m}].$$ Since   $R(Y; S,f^*H_1, \dots, f^*H_k )$ is finitely generated and $$\{S,f^*H_1, \dots, f^*H_k\}$$ is a basis of $\mathrm{Pic}_{\mathbb{Z}}\,Y$, we have that  $Y$ is also a Mori dream space.
\end{proof}

\subsection{Varieties of Fano type and of Calabi--Yau type}

\begin{defi}[cf. {\cite[Lemma-Definition 2.6]{ps}}]\label{Fano pair}Let $X$ be a projective normal variety  and $\Delta$ be an effective $\mathbb{Q}$-divisor on $X$ such that $K_X+\Delta$ is $\mathbb{Q}$-Cartier. 
\begin{enumerate}[(i)]
\item We say that $(X,\Delta)$ is a {\em log Fano pair} if $-(K_X+\Delta)$ is ample and $(X, \Delta)$ is klt. 
We say that $X$ is of {\em Fano type} if there exists an effective $\mathbb Q$-divisor $\Delta$ on $X$ such that $(X,\Delta)$ is a log Fano pair. 
\item  We say that $(X,\Delta)$ is a {\em log Calabi--Yau pair} if $K_X+\Delta  \sim_{\mathbb{Q}} 0$ and $(X, \Delta)$ is log canonical. 
We say that $X$ is of \textit{Calabi--Yau type} if 
there exits an effective $\mathbb{Q}$-divisor $\Delta$ such that $(X, \Delta)$ is log Calabi--Yau.  
\end{enumerate}
\end{defi}

\begin{rem}\label{rm-weakfano} If there exists an effective $\mathbb{Q}$-divisor $\Delta$ on $X$ such that $(X,\Delta)$ is klt and $-(K_X+\Delta)$ is nef and big, then $X$ is of Fano type. See \cite[Lemma-Definition 2.6]{ps}.
\end{rem}

Similarly to Proposition \ref{split-mds}, we have: 

\begin{prop}\label{split vec}
Let $X$ be a projective variety of Fano type and let $L_1,\dots,L_m$ be Cartier divisors  on $X$. 
Let 
$$\mathcal E=\bigoplus_{i=1}^m \mathcal O_X(L_i)$$
and let $f\colon Y=\mathbb P(\mathcal E)\to X$ be the corresponding projective bundle. 

Then $Y$ is also of Fano type 
\end{prop}

\begin{proof}
By assumption, there exist a klt pair $(X,\Delta)$ such that $-(K_X+\Delta)$ is ample. We may assume that $\Delta=A+B$ where $A$ is an ample $\mathbb Q$-divisor and $(X,B)$ is klt. 

After tensoring by a suitable ample divisor on $X$, we may assume that $L_1,\dots,L_m$ are very ample divisors on $X$. Let $S$ be the tautological section of $\mathcal E$, i.e. the section corresponding to $\mathcal O_Y(1)$. Since $L_1,\dots,L_m$ are ample, it follows that $S$ is ample on $Y$. 

Let $T_1,\dots,T_m$ be the divisors in $Y$ given by the summands of $\mathcal E$. 
Then 
 $$2S\sim \sum_{i=1}^m (T_i +f^*L_i).$$

Fix  a sufficiently small rational number $\delta >0$ such that $A-\delta\sum_{i=1}^m f^*L_i$ is ample.
Then for any sufficiently small rational number $\varepsilon >0$, there exists a $\mathbb Q$-divisor $\Gamma\ge 0$  on $Y$ such that
$(Y,\Gamma)$ is klt and 
\begin{eqnarray*}
\sum_{i=1}^m T_i -\varepsilon S + f^*\Delta &=&\sum_{i=1}^m T_i +f^*A -\varepsilon S + f^*B   \\
  &\sim_{ \mathbb Q}&  {{(1-\delta)\sum_{i=1}^m T_i + (2\delta - \varepsilon)S +f^*(A-\delta\sum_{i=1}^m L_i )+f^*B  }}\\
  & \sim_{ \mathbb Q}&\Gamma
 \end{eqnarray*}
 In addition, we have
$$K_Y= -\sum_{i=1}^m T_i+f^*K_X$$
and it follows that
$$-(K_Y+\Gamma)\sim_{\mathbb Q} -(K_Y+\sum_{i=1}^m T_i )-f^*\Delta+\varepsilon S=-f^*(K_X+\Delta)+\varepsilon S.$$
Thus, if $\varepsilon$ is sufficiently small, then $(Y,\Gamma)$ is log Fano and the claim follows. 
\end{proof}

\begin{prop}\label{split-cy}
Let $X$ be a projective variety of Calabi--Yau type and let $L_1,\dots,L_m$ be Cartier divisors on $X$. 
Let 
$$\mathcal E=\bigoplus_{i=1}^k \mathcal O_X(L_i)$$
and let $f\colon Y=\mathbb P(\mathcal E)\to X$ be the corresponding projective bundle. 

Then $Y$ is also of Calabi--Yau type 
\end{prop}

\begin{proof}
By assumption, there exist a log canonical pair $(X,\Delta)$ such that $K_X+\Delta \sim_{\mathbb Q} 0$. 

Let $T_1,\dots,T_m$ be the divisors in $Y$ given by the summands of $\mathcal E$. Then there exist a $\mathbb Q$-divisor $\Gamma\ge 0$  on $Y$ such that
$(Y,\Gamma)$ is log canonical and 
$$\Gamma\sim_{\mathbb Q}\sum_{i=1}^m T_i  + f^{*}\Delta .$$
 In addition, we have
$$K_Y= -\sum_{i=1}^m T_i+f^{*}K_X$$
and it follows that
$$K_Y+\Gamma\sim_{\mathbb Q}f^*(K_X+\Delta).$$
Thus, $(Y,\Gamma)$ is log Calabi--Yau and the claim follows. 
\end{proof}

\section{The anti-canonical model of a  variety of Fano type }\label{s_proof}

We begin with a generalisation of Theorem \ref{t-intro}:

\begin{thm}\label{thm1}Let $(X,\Delta)$ be a projective log pair over an algebraic closed field $k$ of characteristic zero, 
such that $K_X+\Delta$ is $\mathbb Q$-Cartier.

Then there exists a $
\mathbb Q$-divisor $\Delta'\ge 0$ such that $(X,\Delta+\Delta')$ is  a log Fano pair if and only if   $R(X, -m(K_X+\Delta))$ is a finitely generated $k$-algebra for some $m \in \mathbb{N}$, 
the map 
$$\varphi: X \dashrightarrow Y:= \mathrm{Proj}\,R(X,-m(K_X+\Delta))$$ is birational and the pair $(Y, \Gamma=\varphi_*
\Delta)$ is klt. 

\end{thm}
\begin{proof} 
We first prove the "only if" part. To this end, we follow some of the arguments in \cite{gost}. 
Assume that the pair $(X,\Delta+\Delta')$ is log Fano, for some $\mathbb Q$-divisor $\Delta'\ge 0$.
After possibly replacing $X$ by its $\mathbb Q$-factorization  \cite[Corollary 1.4.3]{bchm}, we may assume that $X$ is $\mathbb Q$-factorial. 
In particular $(X,\Delta+\Delta')$ is klt and $A:=-(K_X+\Delta+\Delta')$ is 
ample. Thus $R(X,-m(K_X+\Delta))$ is finitely generated for some positive integer $m$. In addition, $-(K_X+\Delta)$ is big and, in particular, $\varphi$ is a birational map. 
Let $D$ be an effective $\mathbb Q$-divisor such that $D\sim_{\mathbb Q}-(K_X+\Delta)$ and pick a rational number  $0< \varepsilon \ll 1$ such that $(X,\Delta+\Delta'+\varepsilon D)$ is klt. Then
$$\varepsilon D= K_X+\Delta+\Delta'+A+\epsilon D$$
and by  \cite{bchm}, we can run a $D$-MMP and obtain a contraction map $\psi\colon X\dashrightarrow X'$ such that 
 $\psi_* D$ is big and semi-ample.  In particular $(X',\psi_*\Delta)$ is klt. Note that $\varphi$ factors through $\psi$. Let $\eta\colon X'\to Y$ be the induced morphism. Then, $(Y,\Gamma=\eta_*\psi_*\Delta)$ is klt.

Now we show the  "if" part. By assumption 
$$\varphi: X \dashrightarrow Y:= \mathrm{Proj}\,R(X,-m(K_X+\Delta))$$
is a birational contraction and $-(K_Y+\Gamma)=-\varphi_*(K_X+\Delta)$ is $\mathbb Q$-Cartier  
by \cite[Lemma 1.6]{hk}. The Lemma  also implies that  $-(K_Y+\Gamma)$ is ample and 
$$-(K_X+\Delta) = -\varphi^*(K_{Y}+\Gamma) +E,$$
where $E\ge 0$ is $\varphi$-exceptional, and  $\varphi^*$ is the pullback by birational contractions. 
 Thus, it follows that  the discrepancies of  $(Y,\Gamma)$ along the components of $\varphi$-divisors are at most $0$. 
 Since $(Y,\Gamma)$ is klt, by  \cite[Corollary 1.4.3]{bchm} applied to 
 $$\mathfrak{E}=\{F  \mid F  \text{ is a $\varphi$-exceptional divisors on $X$}\}$$ 
 there exists a  birational morphism $f:Y' \to Y$ satisfying:

\begin{enumerate}
\item[(1)] $Y'$ is a $\mathbb{Q}$-factorial normal variety,
\item[(2)]\label{l2}  $Y'$ and $X$ are isomorphic in codimension one, and
\item[(3)] there exists a $\mathbb Q$-divisor $\Gamma'$ such that $f^*(K_Y+\Gamma)=K_{Y'}+\Gamma'$ and $(Y',\Gamma')$ is klt.
  \end{enumerate}
   In particular $Y'$ is of Fano type.  Thus, there exists a big $\mathbb Q$-divisor $\Theta\ge 0$ such that $(Y',\Gamma'+\Theta)$ is klt and $K_{Y'}+\Gamma'+\Theta\sim_{\mathbb Q}0$.

 Let $\Theta'$ be the strict transform of $\Theta$ in $X$. Then $(X,\Delta+\Theta')$ is klt, $\Theta'$ is big and $K_X+\Delta+\Theta'\sim_{\mathbb R}0$. Thus, it follows easily from (2) that there exists a $\mathbb Q$-divisor $\Delta'$ such that $(X,\Delta+\Delta')$ is  a log Fano pair  (e.g. see \cite[Lemma\,2.4]{bir-fano}). 
\end{proof}

\medskip 

As an immediate consequence of the previous Theorem, we obtain:
\begin{cor}\label{char by lt-ness}Let $X$ be a $\mathbb{Q}$-Gorenstein normal variety.

Then $X$ is of Fano type if and only if   $-K_X$ is big, $R(-K_X)$ is finitely generated, and $\mathrm{Proj}\, R(-K_X)$ has log terminal singularities.
\end{cor}

\begin{rem}
Recall that if $X$ is a normal variety, $\Delta$ is an effective $\mathbb Q$-divisor such that $L=r(K_X+\Delta)$ is ample  for some $r\in \mathbb Q$ and $C$ is the cone defined by $L$, then $(X,\Delta)$ is log Fano (in particular $r<0$) if and only if $(C,\Gamma)$ is klt, where $\Gamma$ is the corresponding $\mathbb Q$-divisor on $C$ (e.g. see \cite[Lemma 3.1]{kk13}).
\end{rem}

\medskip

An analogue result as Theorem \ref{thm1} holds for log Calabi-Yau pairs:

\begin{thm}\label{thm2}
Let $(X,\Delta)$ be a projective log pair over an algebraic closed field $k$ of characteristic zero. 
Assume that $R(X, -m(K_X+\Delta))$ is a finitely generated $k$-algebra for some $m \in \mathbb{N}$ and that the map 
$$\varphi: X \dashrightarrow Y:= \mathrm{Proj}\,R(X,-m(K_X+\Delta))$$
is a birational contraction.   

Then, there exists a $\mathbb Q$-divisor $\Delta'\ge 0$ such that $(X,\Delta+\Delta')$ is  a log Calabi--Yau pair if and only if 
the pair $(Y, \Gamma=\varphi_*\Delta)$ has only log canonical singularities. 

\end{thm}

We begin with the following lemma:

\begin{lem}\label{lem-crep}  Let $\mathbb{K}=\mathbb{R}$ or $\mathbb{Q}$. 
Let $(X, \Delta)$ be a log canonical projective  Calabi--Yau pair, for some $\mathbb K$-divisor $\Delta\ge 0$  and let $X'$ be a normal projective variety.  Suppose that there exists  $\varphi:X \dashrightarrow X'$ such that $\varphi$ is a birational  contraction. Then $K_{X'}+\varphi_* \Delta$ is $\mathbb{K}$-Cartier and  $(X', \varphi_*\Delta)$ is  a log canonical Calabi--Yau pair.
\end{lem}

\begin{proof}Take a log resolution $f:W \to X$ of $(X,\Delta)$ which resolves the singularities of $\varphi$ and let $g\colon W \to X'$ be the induced morphism.  
Let $\Gamma$ be a $\mathbb K$-divisor on $W$ such that  $K_W+\Gamma=f^*(K_X+\Delta) \sim_{\mathbb{K}} 0$.  Note that $g_*\Gamma =\varphi_*\Delta$. Thus  $K_{X'}+\varphi_* \Delta$ is $\mathbb{K}$-Cartier and $\mathbb{K}$-linearly trivial. Let $\Gamma'$ be the $\mathbb K$-divisor on $W$ such that  $K_W+\Gamma'=f^*(K_{X'}+\varphi_*\Delta) \sim_{\mathbb{K}} 0$. Thus, by the negativity lemma it follows that $\Gamma=\Gamma'$ and in particular $(X',\varphi_* \Delta)$ is log canonical.
\end{proof}

\begin{proof}[Proof of Theorem \ref{thm2}]The "only if" part is an immediate consequence of the above lemma. Thus, we show the "if" part. 
As in the proof of Theorem \ref{thm1}, we have that $-(K_Y+\Gamma)$ is ample and 
$$-(K_X+\Delta) = -\varphi^*(K_{Y}+\Gamma) +E,$$
where $E\ge 0$ is a $\varphi$-exceptional $\mathbb Q$-divisor.  Note that  the discrepancies of  $(Y,\Gamma)$ with respect to any $\varphi$-exceptional prime divisor is  at most $0$.
Let  $\alpha:W \to Y$ be a log resolution of $(Y,\Gamma)$ which resolves the indeterminacy of $\varphi$ and let $\beta \colon W \to X'$ be the induced morphism. 
By assumption $(Y,\Gamma)$ is log canonical. We write $\alpha^*(K_{Y}+\Gamma)=K_W+\Omega$, where the coefficient of $\Omega$ are less or equal to $1$.
Let $\mathcal{E}$ be the set consisting of all the $\alpha$-exceptional divisors on $W$  which are not $\beta$-exceptional  and all the components of $\Omega$ with coefficient equal to one. 
By  \cite[Theorem 4.1]{fujino-ss}    applied to $\mathcal E$, 
there exists a  birational morphism $f:Y' \to Y$ satisfying:

\begin{itemize}
\item[(1)] $Y'$ is a $\mathbb{Q}$-factorial normal variety,
\item[(2)] $Y' \dashrightarrow X$ is a birational contraction, and  
\item[(3)] there exists a $\mathbb Q$-divisor $\Gamma'$ such that $g^*(K_Y+\Gamma)=(K_{Y'}+\Gamma')$ and $(Y',\Gamma')$ is log canonical and $-(K_{Y'}+\Gamma')$ is semi-ample.
\end{itemize}
In particular $Y'$ is of Calabi--Yau type. Thus, Lemma \ref{lem-crep} implies that there exists  $\Delta'\ge 0$ such that $(X,\Delta+\Delta')$ is  a log Calabi--Yau pair.
\end{proof}

Similarly to Corollary \ref{char by lt-ness}, we immediately obtain:
\begin{cor}
Let $X$ be a $\mathbb{Q}$-Gorenstein normal variety such that $-K_X$ is big and $R(-K_X)$ is finitely generated. Then $X$ is of Calabi--Yau type if and only if  $\mathrm{Proj}\, R(-K_X)$ has log canonical singularities.
\end{cor}

\medskip

In view of  Corollary \ref{char by lt-ness}, it is tempting to ask if for any variety $X$ of Calabi--Yau type and such that $-K_X$ is big, we have that $R(X,-K_X)$ is finitely generated. However, the example below shows that this is not true in general:

\begin{eg}\label{eg1}Let $S$ be the blow up of $\mathbb P^2$ at $9$ very general points and let $E=-K_S$.
Let $L_1$ be and  ample divisor  and let  $L_2=E-L_1$. Define  $X=\mathbb{P}(\mathcal{O}_S(L_1)\oplus \mathcal{O}_ S(L_2))$. Then Proposition\,\ref{split-cy} implies that $X$ is of  Calabi--Yau type. 
On the other hand, we have that $\mathcal O_X(-K_X)=\mathcal O_X(2)$ is not semi-ample since $E$ is not semi-ample. 
Thus, since $-K_X$ is big, it follows that  $R(X,-K_X)=R(X,\mathcal{O}_X(2))$   is not finitely generated.
\end{eg}

\section{On log canonical Fano varieties}\label{lc}

Considering Example \ref{eg1}, it seems worth studying under which condition a projective variety $X$ is a Mori dream space. 
 By \cite[Corollary 1.3.2]{bchm}, it follows that any variety of Fano type is a Mori dream space. The aim of this section is to study possible generalisation of this result. 

We begin with the following conjecture:

\begin{conj}\label{prob19}Let $(X,\Delta)$ be a $\mathbb{Q}$-factorial log canonical pair  such that $-(K_X+\Delta)$ is ample. Then $X$ is a Mori dream space.
\end{conj} 
Note that  by \cite[Corollary 1.3.2]{bchm},
 the conjecture holds to be true under the assumption that $(X,\Delta)$ is dlt. 
We now show that the conjecture follows from the standard conjectures of the minimal model program:

\begin{thm}\label{thm3}Assme  the existence of minimal models and the abundance conjecture for log canonical pairs  in dimension $n$. 

Then Conjecture \ref{prob19} is true in dimension $n$.
\end{thm}

\begin{proof}
First, note that  Kodaira vanishing  for log canonical pairs (e.g. see \cite{ambro-quasi-log}   and \cite[Theorem 2.42]{fujino-book}) implies that $H^1(X, \mathcal{O}_X)=0$. 
Thus $\mathrm{Pic}_{\mathbb{Q}}(X)\simeq N^1_{\mathbb{Q}}(X)$ (eg. see \cite[Corollary 9.5.13]{fga-ex}).  
Let $L_1,\dots,L_\ell$ be a basis of $\mathrm{Pic}_{\mathbb{Q}}(X)$. Then,  for any $i=1,\dots,\ell$, after possibly rescaling $L_i$, there exist effective $\mathbb{Q}$-divisors $\Delta_i$ such that $(X,\Delta_i)$ is log canonical and $L_i \sim_{\mathbb{Q}} K_X+\Delta_i$.

We claim that the ring
$$\bigoplus_{(m_1,\dots, m_l) \in \mathbb{N}}H^{0}(X, \sum_{i=1}^l m_i m(K_X+\Delta_i))$$
s  finitely generated for any 
positive integer $m$ such that $m(K_X+\Delta_i)$ is Cartier for any $i=1,\dots,\ell$. 
Note that the claim implies that $X$ is a Mori dream space by \cite[Theorem 2.13]{hk}.

We now proceed with the proof of the claim. After taking a log resolution, we may assume that $(X,\Delta_i)$ is log smooth for any $i=1,\dots,\ell$. 
Let $V\subseteq \mathrm{Div}_{\mathbb R}(X)$ be the subspace spanned by the components of $\Delta_1,\dots,\Delta_\ell$. By \cite[Theorem 3.4]{SC} or \cite[Theorem 3]{ka}, there exist $\varphi_j\colon X\dashrightarrow Y_i$, with $j=1,\dots,p$ such that if $\Delta\in V$ is such that $(X,\Delta)$ is log canonical and $K_X+\Delta$ is pseudo-effective, then $\varphi_j$ is the log terminal model of $(X,\Delta)$ for some $j=1,\dots,p$. Thus, we obtain the claim, by using the same methods as in the proof of 
\cite[Corollary 1.1.9]{bchm}. 
\end{proof}

In particular, we obtain:

\begin{cor}Let $(X,\Delta)$ a $3$-dimensional  log canonical pair  such that $-(K_X+\Delta)$ is ample. Then $X$ is a Mori dream space.

\end{cor}

Note that, in general,  a log canonical weak Fano pair is not necessarily a  Mori dream space (see \cite[Section 5]{g1}).  The following example shows that even assuming that $(X,\Delta)$ is a log canonical pair such that  $-(K_X+\Delta)$ is semi-ample and big,  then $X$ is not necessarily a Mori dream space, even assuming $H^1(X, \mathcal{O}_X)=0$.

\begin{eg}\label{rem5} Let $S$ be a K3 surface which is not a Mori dream space. Let $A$ be a very ample divisor on $S$  and let $X=\mathbb{P}(\mathcal{O}_S \oplus \mathcal{O}_S(A))$.  Thus, \cite[Theorem 1.1]{okawa} implies that $X$ is not a Mori dream space. On the other hand, it is easy to check that if $E$ is the section corresponding to the summand $\mathcal O_S$ then $(X,E)$ is log canonical and  $-(K_X+E)$ is semi-ample and big. Furthermore, we have $H^1(X,\mathcal O_X)=0$.
\end{eg}

\section{On varieties of Fano type}\label{questions}

The purpose to this section is to explore possible characterizations of varieties of Fano type. 
We show two examples which show that some of the questions raised before about this problem have a negative answer.

\begin{eg}\label{eg-prob2} Assume that $(X,\Delta)$ is a kawamata log terminal pair such that $-(K_X+\Delta)$ is movable and big. Then $X$ is not necessarily of Fano type.

Indeed, let $S$ be a  smooth projective variety of general type such that $H^1(S,\mathcal{O}_X)=0$ and such that the Picard number of $S$ is one. Let $L_1$ and $L_2$ be ample divisors on $S$, and let 
$$\mathcal E=\mathcal O_S(L_1)\oplus \mathcal O_S(L_2) \oplus \mathcal O_S(-K_S-L_1-L_2).$$
Then the anti-canonical divisor of $X:=\mathbb{P}_S(E)$ 
is movable and big but since $X$ is not rationally connected, it follows that $X$ is not of Fano type.
Note that Proposition \ref{split-mds} implies that $X$ is a Mori dream space.
 \end{eg}

The following question is motivated by the birational Borisov-Alexeev-Borison conjecture. In \cite{okada}, Okada shows that the class of rationally connected varieties and of Fano type are birationally unbounded. Thus, it is natural to ask if these classes are the same up to birational equivalence.  Recall that any variety of Fano type is rationally connected (see \cite{zhang} and \cite{hacon-mckernan-shok-rat}) and  all smooth rationally connected surface are rational.

\begin{ques}\label{prob3}
Let $X$ be a rationally connected variety. Then is $X$ birationally equivalent to a variety of Fano type?
 \end{ques}

\providecommand{\bysame}{\leavevmode \hbox \o3em
{\hrulefill}\thinspace}

\end{document}